\documentclass[oneside,11pt]{amsart}
\usepackage{amsmath}
\usepackage{amssymb}
\usepackage{graphicx}
\usepackage{caption}
\usepackage{subcaption}
\usepackage{pinlabel}

\newtheorem{thm}{Theorem}[section]
\newtheorem{prop}[thm]{Proposition}
\newtheorem{lem}[thm]{Lemma}

\newtheorem*{main}{Main Theorem}

\theoremstyle{definition}
\newtheorem{defn}[thm]{Definition}
\newtheorem{rem}[thm]{Remark}
\newtheorem{conv}[thm]{Convention}

\newtheorem{ques}[thm]{Question}

\newcommand{\abs}[1]{\lvert{#1}\rvert}

\DeclareMathOperator{\CAT}{CAT}

\DeclareMathOperator{\Div}{Div}
\DeclareMathOperator{\ldiv}{ldiv}

\begin{document}

\title[Divergence of Morse geodesics]{Divergence of Morse geodesics}

\author{Hung Cong Tran}
\address{Dept.\ of Mathematical Sciences\\
University of Wisconsin--Milwaukee\\
P.O.~Box 413\\
Milwaukee, WI 53201\\
USA}
\email{hctran@uwm.edu}

\date{\today}

\begin{abstract}
Behrstock and Dru\c{t}u raised a question about the existence of Morse geodesics in $\CAT(0)$ spaces with divergence function strictly greater than $r^n$ and strictly less than $r^{n+1}$, where $n$ is an integer greater than $1$. In this paper, we answer the question of Behrstock and Dru\c{t}u by showing that for each real number $s\geq 2$, there is a $\CAT(0)$ space $X$ with a proper and cocompact action of some finitely generated group such that $X$ contains a Morse bi-infinite geodesic with the divergence equivalent to $r^s$.
\end{abstract}

\subjclass[2000]{%
20F67, 
20F65} 
\maketitle

\section{Introduction}
The \emph{divergence} of two geodesic rays $\alpha$ and $\beta$ with the same initial point $x_0$ in a geodesic space $X$, denoted $\Div_{\alpha,\beta}$, is the function $g: (0,\infty)\to(0,\infty]$ defined as follows. For each positive number $r$ the value $g(r)$ is $\infty$ if there is no path outside the open ball with radius $r$ about $x_0$ connecting $\alpha(r)$ and $\beta(r)$ or otherwise the value $g(r)$ is the infimum on the lengths of all paths outside the open ball with radius $r$ about $x_0$ connecting $\alpha(r)$ and $\beta(r)$. Consequently, the \emph{divergence of a bi-infinite geodesic} $\gamma$, denoted $\Div_{\gamma}$, is the divergence of the two geodesic rays obtained from $\gamma$ with the initial point $\gamma(0)$. A (quasi-)geodesic $\gamma$ is \emph{Morse} if for any constants $K>1$ and $L> 0$, there is a constant $M = M(K,L)$ such that every $(K,L)$--quasi-geodesic $\sigma$ with endpoints on $\gamma$ lies in the $M$--neighborhood of $\gamma$. In \cite{BD}, Behrstock and Dru\c{t}u asked a question:


\begin{ques} (see Question 1.5, \cite{BD})
\label{q2}
Can the divergence function of a Morse geodesic in a $\CAT(0)$ space be
greater than $r^{m-1}$ and less than $r^m$ for each $m\geq 3$?
\end{ques}

The main theorem gives a positive answer to the above question as follows:
\begin{main}
For each integer $m\geq 2$, there is a $\CAT(0)$ space $Y_m$ with a proper, cocompact action of some finitely generated group such that for each $s$ in $[2,m]$ there is a Morse geodesic in $Y_m$ with the divergence function equivalent to $r^s$.
\end{main}

In the above theorem, we refer the reader to Convention \ref{conv} for the concept of equivalence. 

The existence of $\CAT(0)$ spaces containing Morse geodesics with polynomials divergence function of degree greater than one can be deduced from the work of Macura \cite{MR3032700}, Dani-Thomas \cite{MR3314816}, Charney-Sultan~\cite{MR3339446}, and Tran~\cite{MR3361149}. However, the existence of Morse geodesic with divergence greater than $r^{m-1}$ and less than $r^m$ for $m\geq 3$ was still mysterious. 

In the main theorem, we use the spaces constructed by Dani-Thomas in \cite{MR3314816} to construct Morse geodesics with desired divergence function. Dani-Thomas built these spaces to study right-angled Coxeter groups with polynomial divergence of arbitrary degree. They developed most of techniques which allow us to construct Morse geodesics with arbitrary polynomials divergence easily. Motivated by their work, we worked on the construction of Morse geodesics with the divergence functions equivalent to $r^s$ for any real number $s\geq 2$. Though we used the techniques of Dani-Thomas in our construction of geodesics, some additional techniques were required to obtain Morse geodesics with divergence function $r^s$ where $s$ is non-integer number greater than 2.     

In \cite{MR1254309}, Gersten introduced the divergence of a space. Consequently, he defined the divergence of a finitely generated group to be the divergence of its Cayley graph. This concept has been studied by Macura \cite{MR3032700}, Behrstock-Charney \cite{MR2874959}, Duchin-Rafi \cite{MR2563768}, Dru{\c{t}}u-Mozes-Sapir \cite{MR2584607}, Sisto \cite{Sisto} and others. Moreover, the divergence of a space is a quasi-isometry invariant, and it is therefore a useful tool to classify finitely generated groups up to quasi-isometry. In the concept of the divergence of a space, Gersten used the concept of the divergence of two geodesic rays as the main idea to define it. Therefore, if we understand the divergence of bi-infinite geodesic, we may also understand the divergence of the whole space as well as a group that acts properly, cocompactly on the space. In particular, we hope that the main theorem can shed a light for the positive answer to the following question: 

\begin{ques} (see Question 1.3, \cite{BD})
\label{q1}
Are there examples of $\CAT(0)$ groups whose divergence in the sense of Gersten is strictly between $r^{m-1}$ and $r^m$ for some $m$?
\end{ques}


The Morse property of quasi-geodesics is a quasi-isometry invariant. Therefore, it is a useful tool to classify geodesic spaces as well as finitely generated groups up to quasi-isometry. The main theorem reveals a geometric aspect of Morse geodesics. Therefore, the main theorem reveals a geometric aspect of Morse quasi-geodesics since each Morse quasi-geodesic has a finite Hausdorff distance from some Morse geodesic. Moreover, the main theorem helps us come up with a new quasi-isometry invariant of spaces, called spectrum divergence. The \emph{spectrum divergence} of a geodesic space is a family $S$ of functions from positive reals to positive reals such that a function $f$ belongs to $S$ if there is a bi-infinite Morse geodesic in the space with divergence function equivalent to $f$. The main theorem reveals that the spectrum divergence of a $\CAT(0)$ space can contain uncountably many of power functions. Moreover, it is still unknown what kind of function besides power functions can belongs to the spectrum of a $\CAT(0)$ space.


\subsection*{Acknowledgments}
I would like to thank my advisor Prof.~Christopher Hruska for very helpful comments and suggestions. I also thank the referee for the helpful advice and the suggestion that improves the main theorem in the article. I want to thank Pallavi Dani, Jason Behrstock, Anne Thomas and Hoang Thanh Nguyen for their helpful correspondences.

\section{Right-angled Coxeter groups}

\begin{defn}
Given a finite, simplicial graph $\Gamma$, the associated right-angled Coxeter group $G_{\Gamma}$ has generating set $S$ the vertices of $\Gamma$, and relations $s^2 = 1$ for all $s$ in $S$ and $st = ts$ whenever $s$ and $t$ are adjacent vertices. 



\end{defn}

\begin{defn}
Given a nontrivial, finite, simplicial, triangle-free graph $\Gamma$ with the set $S$ of vertices, we may define the \emph{Davis complex} $\Sigma=\Sigma_{\Gamma}$ to be the Cayley 2--complex for the presentation of the Coxeter group $G_{\Gamma}$, in which all disks bounded by a loop with label $s^2$ for $s$ in $S$ have been shrunk to an unoriented edge with label $s$. Hence, the vertex set of $\Sigma$ is $G_{\Gamma}$ and the 1-skeleton of $\Sigma$ is the Cayley graph $C_{\Gamma}$ of $G_{\Gamma}$ with respect to the generating set $S$. Since all relators in this presentation other than $s^2 = 1$ are of the form $stst = 1$, $\Sigma$ is a square complex. The Davis complex $\Sigma_{\Gamma}$ is a $\CAT(0)$ space and the group $G_{\Gamma}$ acts properly and cocompactly on the Davis complex $\Sigma_{\Gamma}$ (see \cite{MR2360474}).
\end{defn}

\begin{defn} 
Let $\Gamma$ be a nontrivial, finite, simplicial, triangle-free graph and $\Sigma=\Sigma_{\Gamma}$ the associated Davis complex. We observe that each edge of $\Sigma$ is on the boundary of a square. We define \emph{a midline} of a square in $\Sigma$ to be a geodesic segment in the square connecting two midpoints of its opposite edges. We define \emph{a hyperplane} to be a connected subspace that intersects each square in $\Sigma$ in empty set or a midline. Each hyperplane divides the square complex $\Sigma$ into two components. We define \emph{the support} of a hyperplane $H$ to be the union of squares which contain edges of $H$. For each vertex $a$ of $\Gamma$, we define the subcomplex $H_{a}$ to be the support of the hyperplane that crosses the edge labeled by $a$ with one endpoint $e$.

Since each square in $\Sigma$ has the label of the form $stst$, each midline in each square of $\Sigma$ connects two midpoints of edges with the same label. Thus, each hyperplane is a graph and vertices are the midpoints of edges with the same label. Therefore, we define \emph{the type} of a hyperplane $H$ to be the label of edges containing vertices of $H$. Obviously, if two hyperplanes with the types $a$ and $b$ intersect, then $a$ and $b$ commute. 
\end{defn}

\begin{rem}
The length of a path $\alpha$ in $C_{\Gamma}$ is equal to its number of hyperplane-crossings. A path is a geodesic if and only if it does not cross any hyperplane twice (see Lemma~3.2.14 \cite{MR2360474}).

If $\Gamma'$ is a full subgraph of $\Gamma$, then the Davis complex $\Sigma_{\Gamma'}$ with respect
to the graph $\Gamma'$ embeds isometrically in the Davis complex $\Sigma_{\Gamma}$ with respect
to the graph $\Gamma$.
\end{rem}



\begin{figure}
\centering
\labellist
\small\hair 2pt

\pinlabel $\Gamma_1$ at 100 230
\pinlabel $\Gamma_2$ at 250 230
\pinlabel $\Gamma_3$ at 500 230
\pinlabel $\Gamma_m$ at 975 230

\pinlabel $a_0$ at 830 225
\pinlabel $b_0$ at 830 100
\pinlabel $b_1$ at 775 165
\pinlabel $a_1$ at 885 165
\pinlabel $a_2$ at 925 165
\pinlabel $a_3$ at 975 165
\pinlabel $a_{m-1}$ at 1120 140
\pinlabel $a_m$ at 1170 165
\pinlabel $b_2$ at 830 25
\pinlabel $b_3$ at 880 25
\pinlabel $b_4$ at 930 25
\pinlabel $b_{m-1}$ at 1015 25
\pinlabel $b_m$ at 1065 25
\endlabellist
\includegraphics[scale=0.35]{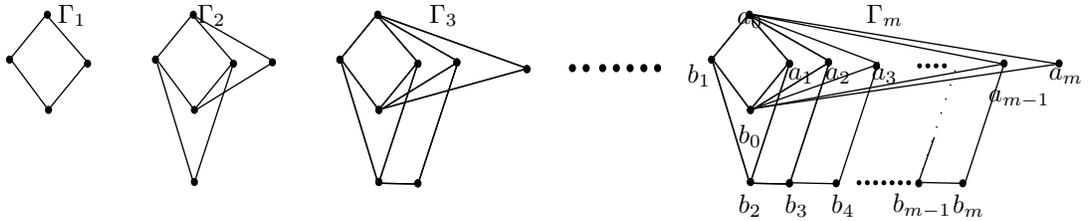}
\caption{A collection of nontrivial, connected, finite, simplicial, triangle-free graphs.}
\label{asecond}
\end{figure}

\section{Proof of the Main Theorem}

\begin{conv}
\label{conv}
Let $f$ and $g$ be two functions from positive reals to positive reals. We say that $f$ is \emph{dominated} by $g$, denoted \emph{$f\preceq g$}, if there are positive constants $A$, $B$, $C$ such that $f(x)\leq g(Ax)+Bx$ for all $x>C$. We say that $f$ is \emph{equivalent} to $g$, denoted $f\sim g$, if $f\preceq g$ and $g\preceq f$. 
\end{conv}

\begin{prop}
\label{mp}
For each positive integer $m$, let $\Gamma_m$ be the graph shown in Figure \ref{asecond} and $X_m$ be the Davis complex with respect to $\Gamma_m$. Then the $\CAT(0)$ space $X_2$ contains a Morse geodesic with divergence $r^2$ and the $\CAT(0)$ space $X_m$ contains a Morse geodesic with divergence $r^s$ for each $m\geq 3$ and $s$ in $(m-1,m]$.
\end{prop}

In the rest of the paper, for each positive integer $m$ we are going to use the notation $X_m$ for the Davis complex with respect to the graph $\Gamma_m$ shown in Figure \ref{asecond}. We remark that graphs $\Gamma_m$ and spaces $X_m$ were first constructed by Dani-Thomas in \cite{MR3314816}, and that we will also use some of their results in the proof of Proposition \ref{mp}. In particular, the existence of Morse geodesics with quadratic divergence in $X_2$ can also be seen in \cite{MR3314816}. In the following definition, we construct Morse bi-infinite geodesics with desired divergence functions.
\begin{defn}
For each integer $m\geq 3$, integer $i\geq 1$ and real number $t>1$, let $w_{m,i,t}$ be the word $(a_m b_m) (a_mb_2)^{\lfloor i^{t-1}\rfloor}$. If $m$ and $t$ are clear from context, we can use the notation $w_i$ instead of $w_{m,i,t}$. Let $\gamma_{t,m}$ be the bi-infinite path in $X_m$ which passes through e and labeled by $\cdots (a_mb_2)(a_mb_2)(a_mb_2)w_1 w_2 w_3\cdots$, such that $\gamma_{t,m}(0)=e, \gamma_{t,m}(4)=w_1$, and $\gamma_{t,m}(-2i)=(a_mb_2)^{-i}$ for each positive integer $i$. 
\end{defn}

We observe that the labels of two consecutive edges of $\gamma_{t,m}$ do not commute. Thus, $\gamma_{t,m}$ is a bi-infinite geodesic (see Theorem 3.4.2, \cite{MR2360474}). We define the function $f_t$ on the set of positive integers as follows: 
\[f_t(n)=\lfloor 1^{t-1}\rfloor+\lfloor 2^{t-1}\rfloor+\lfloor 3^{t-1}\rfloor+\lfloor 4^{t-1}\rfloor+\lfloor 5^{t-1}\rfloor+\dots+\lfloor n^{t-1}\rfloor.\]

There are constants $0< h_t\leq 1/2$ and $n_t>0$ such that for each $n>n_t$ the following holds:
\[h_t n^t \leq f_t(n)\leq n^t.\]

We are going to use the constants $h_t$ and $n_t$ many times in the rest of the paper.

\begin{rem}
For each nontrivial, finite, simplicial, triangle-free graph $\Gamma$, the associated Davis complex $\Sigma_{\Gamma}$ is a $\CAT(0)$ space. The first skeleton $C_{\Gamma}$ of $\Sigma_{\Gamma}$ is a Cayley graph of the group $G_{\Gamma}$ with the word metric. It is not hard to see that the natural embedding of $C_{\Gamma}$ in to $\Sigma_{\Gamma}$ is a quasi-isometry. Moreover, we can see easily that the divergence function of a pair geodesic rays with the same initial point in $C_{\Gamma}$ with respect to the word metric is equivalent to the divergence function of this pair of rays with respect to the $\CAT(0)$ metric.

We remark that the main theorem is stated under the $\CAT(0)$ metric and most results about divergence of pairs of geodesic rays we are going to use from \cite{MR3314816} are stated under the word metric. However, we can apply these results to the case of the $\CAT(0)$ metric by the above observation.
\end{rem}

\begin{rem}
\label{r1}
Let $\alpha$ and $\beta$ be two geodesic rays in a $\CAT(0)$ space with the same initial point $x_0$. Assume that $\Div_{\alpha,\beta}(r)\geq f(r)$. Using the fact that the projection onto a closed ball does not increase distances, we can show that if $\eta$ is a path outside $B(x_0,r)$ connecting two points on $\alpha$ and $\beta$, then $\ell(\eta)\geq f(r)$. These observations will be used sometimes in the rest of the paper.
\end{rem}

We now recall Gersten's definition of divergence from \cite{MR1254309} and get more results on the divergence of pairs of geodesic rays from \cite{MR3314816}.

Let $X$ be a geodesic space and $x_0$ one point in $X$. Let $d_{r,x_0}$ be the induced length metric on the complement of the open ball with radius $r$ about $x_0$. If the point $x_0$ is clear from context, we can use the notation $d_r$ instead of using $d_{r,x_0}$.

\begin{defn}
Let $X$ be a geodesic space with geodesic extension property and $x_0$ one point in $X$. We define the \emph{divergence} of $X$, denoted $Div_X$, as a function $\delta:[0, \infty)\to [0, \infty)$ as follows: 

For each $r$, let $\delta(r)=\sup d_r(x_1,x_2)$ where the supremum is taken over all $x_1, x_2 \in S_r(x_0)$ such that $d_r(x_1, x_2)<\infty$.
\end{defn}

\begin{rem}
\label{rma1}
We remark that the above definition is only applied to geodesic spaces with geodesic extension property (i.e. any finite geodesic segment can be extended to an infinite geodesic ray) and it is a simplified version of the concept of the divergence of geodesic spaces in general (see \cite{MR1254309}). Since the spaces $X$ that we will consider (Cayley graphs of right-angled Coxeter groups or Davis complexes of right-angled Coxeter groups) have the geodesic extension property, the above definition works well for the purpose of this paper. Moreover, it is not hard to see that the divergence of each pair of rays with the same initial point in $X$ must be dominated by the divergence of $X$.
\end{rem}

\begin{lem}
\label{bl0}
Let $m$ be an arbitrary positive integer. For $1 \leq n \leq m$, let $\alpha_n$ and $\beta_n$ be any geodesic rays in $X_{m+2}$ satisfying the
following conditions:
\begin{enumerate}
\item $\alpha_n$ emanates from $e$ and travels along $H_{b_{n+1}}$, and
\item $\beta_n$ emanates from $e$ and travels along one of $H_{a_n}$, $H_{b_n}$, or $H_{b_{n+2}}$.
\end{enumerate}
Then the divergence of the pair of rays $(\alpha_n,\beta_n)$ dominates $r^n$ in $X_{m+2}$.
\end{lem}

The proof of the above lemma was shown in the proof of Proposition 5.3 in \cite{MR3314816}.

\begin{lem}
\label{bl1}
Let $m\geq 3$ be an integer. In the $\CAT(0)$ space $X_m$, let $\alpha$ be an arbitrary geodesic ray emanating from $e$ that travels along $H_{a_m}$ and let $\beta$ be a path emanating from $e$ that travels along $H_{b_m}$. Then the divergence of $\alpha$ and $\beta$ is equivalent to $r^{m-1}$ in $X_m$.
\end{lem}

\begin{proof}
We know that the space $X_m$ embeds isometrically in $X_{m+2}$. We observe that the ray $\alpha$ is labeled by $a_0b_0a_0b_0\cdots$ or $b_0a_0b_0a_0\cdots$ in $X_m$. Similarly, the ray $\beta$ is labeled by $a_{m-1}b_{m-1}a_{m-1}b_{m-1}\cdots$ or $b_{m-1}a_{m-1}b_{m-1}a_{m-1}\cdots$ in $X_m$. Therefore, the ray $\alpha$ also travels along $H_{a_{m-1}}$ in $X_{m+2}$. Moreover, the ray $\beta$ travels along $H_{b_m}$ in $X_{m+2}$. Thus, the divergence of $\alpha$ and $\beta$ dominates $r^{m-1}$ in $X_{m+2}$ by Lemma~\ref{bl0}. Since any path that avoids a ball centered at some point $x_0$ with some radius $r$ in $X_m$ still avoids the ball at $x_0$ with the same radius $r$ in $X_{m+2}$, the divergence of $\alpha$ and $\beta$ also dominates $r^{m-1}$ in $X_m$.

Since $\alpha$ and $\beta$ are also geodesic rays in $X_{m-1}$ and the divergence of $X_{m-1}$ is $r^{m-1}$ (see \cite{MR3314816}), the divergence of $\alpha$ and $\beta$ is dominated by $r^{m-1}$ in $X_{m-1}$. Again, the space $X_{m-1}$ embeds isometrically in $X_m$. Therefore, any path that avoids a ball centered at some point $x_0$ with some radius $r$ in $X_{m-1}$ still avoids the ball at $x_0$ with the same radius $r$ in $X_m$. This implies that the divergence of $\alpha$ and $\beta$ also dominates $r^{m-1}$ in $X_m$. Therefore, the divergence of $\alpha$ and $\beta$ is equivalent to $r^{m-1}$ in $X_m$.
\end{proof}

\begin{lem}
\label{bl2}

Let $m\geq 3$ be an integer. In the $\CAT(0)$ space $X_m$, let $\alpha$ be an arbitrary geodesic ray emanating from $e$ that travels along $H_{a_m}$ and let $\alpha'$ be a path emanating from $e$ consisting of a geodesic segment labeled $a_mb_m$ followed by an arbitrary geodesic ray emanating from $a_mb_m$ that travels along $a_mb_mH_{a_m}$. Then $\alpha'$ is a geodesic ray and the divergence of the pair $(\alpha,\alpha')$ is equivalent to $r^{m-1}$.
\end{lem}
 
\begin{proof}
It is obvious that each pair of two consecutive generators of $\alpha'$ do not commute. Therefore, $\alpha'$ is a geodesic ray. Let $\alpha_1$ be the geodesic ray emanating from $a_m$ with the same label as $\alpha$. Therefore, the two rays $\alpha$ and $\alpha_1$ both lie in the support of the hyperplane labeled by $a_m$ and the Hausdorff distance between them is exactly 1. Let $\alpha_2$ be a sub-ray of $\alpha'$ with the initial point $a_mb_m$. Let $\beta_1$ and $\beta_2$ be two arbitrary geodesic rays in the support of the hyperplane labeled by $b_m$ such that the initial point of $\beta_1$ is $a_m$ and the initial point of $\beta_2$ is $a_mb_m$. By Lemma~\ref{bl1}, the divergence functions of the two pairs $(\alpha_1, \beta_1)$ and $(\alpha_2, \beta_2)$ are both equivalent to $r^{m-1}$. 

We now prove that the divergence of the pair $(\alpha,\alpha')$ dominates $r^{m-1}$. Let $r$ be an arbitrary number greater than 2. Let $\gamma$ be an arbitrary path outside $B(e,2r)$ connecting $\alpha(2r)$ and $\alpha'(2r)$. Obviously, $\gamma$ also lies outside $B(a_mb_m,r)$. We see that the path $\gamma$ must cross $a_mH_{b_m}$. Thus, the path $\gamma$ must cross some geodesic rays $\beta_2$ emanating from $a_mb_m$ that travels along $a_mH_{b_m}$. Therefore, a subpath of $\gamma$ connects two points on the geodesic rays $\alpha_2$, $\beta_2$ and it lies outside $B(a_mb_m,r)$. This implies that the length of $\gamma$ is at least $\Div_{\alpha_2,\beta_2}(r)$ by Remark \ref{r1}. Therefore, the divergence function $\Div_{\alpha,\alpha'}(2r)$ must be greater than $\Div_{\alpha_2,\beta_2}(r)$. Also, the divergence function $\Div_{\alpha_2,\beta_2}$ is equivalent to $r^{m-1}$ by the above observation. Thus, the divergence of the pair $(\alpha,\alpha')$ must dominate $r^{m-1}$.

We will finish the proof by showing that the divergence of the pair $(\alpha,\alpha')$ is dominated by $r^{m-1}$. Let $r$ be an arbitrary number greater than 2. Let $\beta_1$ and $\beta_2$ be two geodesic rays with the same labels in the support of the hyperplane labeled by $b_m$ such that the initial point of $\beta_1$ is $a_m$ and the initial point of $\beta_2$ is $a_mb_m$. Thus, we can connect $\beta_1(2r)$ and $\beta_2(2r)$ by an segment $\gamma_1$ of length 1. Similarly, we can connect $\alpha(2r)$ and $\alpha_1(2r)$ by an segment $\gamma_2$ of length 1. Let $\eta_1$ be a path outside $B(a_m,2r)$ connecting $\alpha_1(2r)$ and $\beta_1(2r)$ with length at most $\Div_{\alpha_1,\beta_1}(2r)+1$. Let $\eta_2$ be a path outside $B(a_mb_m,2r)$ connecting $\alpha_2(2r)$ and $\beta_2(2r)$ with length at most $\Div_{\alpha_2,\beta_2}(2r)+1$. Let $\ell_1$ be the subsegment of $\alpha$ connecting $\alpha(r)$ and $\alpha(2r)$. Let $\ell_2$ be the subsegment of $\alpha$ connecting $\alpha_2(2r)$ and $\alpha_2(r-2)=\alpha'(r)$. Concatenating the paths $\ell_1$, $\gamma_2$, $\eta_1$, $\gamma_1$, $\eta_2$, and $\ell_2$, we have a path outside $B(e,r)$ connecting $\alpha(r)$ and $\alpha'(r)$ with length at most $\Div_{\alpha_1,\beta_1}(2r)+\Div_{\alpha_2,\beta_2}(2r)+2r+6$. Therefore, the divergence function of the pair $(\alpha,\alpha')$ is also at most $\Div_{\alpha_1,\beta_1}(2r)+\Div_{\alpha_2,\beta_2}(2r)+2r+6$. Since both divergence functions $\Div_{\alpha_1,\beta_1}$ and $\Div_{\alpha_2,\beta_2}$ are both equivalent to $r^{m-1}$, the divergence function of the pair $(\alpha,\alpha')$ is dominated by $r^{m-1}$.
\end{proof}

\begin{lem}
\label{bl3}

Let $m\geq 3$ be an integer. In the $\CAT(0)$ space $X_m$, let $\alpha$ be an arbitrary geodesic ray emanating from $e$ that travels along $H_{a_m}$ and let $\alpha'$ be a path emanating from $e$ consisting of a geodesic segment labeled $a_mb_2$ followed by an arbitrary geodesic ray emanating from $a_mb_2$ that travels along $a_mb_2H_{a_m}$. Then $\alpha'$ is a geodesic ray and the divergence of the pair $(\alpha,\alpha')$ is linear. Moreover, the union of the two rays $\alpha$ and $\alpha'$ is a quasi-geodesic. 
\end{lem}
 
\begin{proof}
We observe that each pair of two consecutive generators of $\alpha'$ do not commute. Thus, $\alpha'$ is a geodesic ray. Let $\alpha_1$ be an arbitrary geodesic ray emanating from $a_m$ with the same label as $\alpha$. Therefore, the two rays $\alpha$ and $\alpha_1$ both lie in the support of the hyperplane labeled by $a_m$ and the Hausdorff distance between them is exactly 1. Let $\alpha_2$ be a sub-ray of $\alpha'$ with the initial point $a_mb_2$. Let $\beta_1$ and $\beta_2$ be two geodesic rays labeled by $b_1a_1b_1a_1\cdots$ such that the initial point of $\beta_1$ is $a_m$ and the initial point of $\beta_2$ is $a_mb_2$. Thus, the two rays $\beta_1$ and $\beta_2$ both lie in the support of the hyperplane labeled by $b_2$ and the Hausdorff distance between them is exactly 1. Since the subgroup generated by $\{a_0,b_0,a_1,b_1\}$ is one-ended virtually abelian, the associated Davis complex $\Sigma'$ has linear divergence. Also the pair of rays $(\alpha_1, \beta_1)$ lies in some translation of $\Sigma'$, therefore the divergence of this pair is at most linear. Similarly, the divergence of the pair $(\alpha_2, \beta_2)$ is also at most linear. Therefore, the divergence of the pair $(\alpha,\alpha')$ is dominated by a linear function.

Let $\alpha_3$ be a union of the ray $\alpha_1$, the edge $b_2$, and the ray $\alpha_2$. We observe that each pair of two consecutive generators of $\alpha_3$ do not commute. Therefore, $\alpha_3$ is a geodesic ray. Therefore, it is not hard to see the union of the two rays $\alpha$ and $\alpha'$ is a quasi-geodesic. 
\end{proof}

\begin{prop}
For each integer $m\geq 3$ and real number $t>1$, the divergence $\Div_{\gamma_{t,m}} \preceq r^{m-1+\frac{1}{t}}$ in $X_m$.
\end{prop}

\begin{proof}
For each number $r$ large enough, we can choose an integer $n>n_t$ such that
\[r\leq h_t n^t \leq f_t(n) \leq n^t\leq \biggl(\frac{2}{h_t}\biggr)r.\]

Let $x=w_1 w_2 w_3 w_4\cdots w_n$. Then
\begin{align*}
\abs{x} &= 2\bigl((\lfloor 1^{t-1}\rfloor+1)+(\lfloor 2^{t-1}\rfloor+1)+(\lfloor 3^{t-1}\rfloor+1)+(\lfloor 4^{t-1}\rfloor+1)+(\lfloor 5^{t-1}\rfloor+1)+\cdots\\&\cdots+(\lfloor n^{t-1}\rfloor+1)\bigr)=2n+2f_t(n).
\end{align*}
Thus, $2 f_t(n) \leq \abs{x} \leq 4 f_t(n)$. Therefore, we can connect $x$ and $\gamma_{t,m}(r)$ by a path $\beta_1$ outside $B(e,r)$ such that
\[\ell(\beta_1)\leq 4 f_t(n)-r\leq \biggl(\frac{8}{h_t}-1\biggr)r.\] 

We now try to connect $\gamma_{t,m}(-r)$ and $x$ by a path $\beta_2$ outside $B(e,r)$ such that $\ell(\beta_2)\leq M r^{m-1+\frac{1}{t}}$ for some constant $M$ not depending on $r$ and which completes the proof of the proposition. 

Let $k=t-1$, $\ell=\abs{x}/2$, and $s_i=\gamma_{t,m}(2i)$ for each $0\leq i\leq \ell$. Recall that $H_{a_m}$ is the support of the hyperplane that crosses the edge $a_m$ with one endpoint $e$. For $0\leq i \leq \ell$, let $u_i$ be a geodesic ray which runs along the support $s_iH_{a_m}$ with the initial point $s_i$. We can choose $u_i$ such that they have the same label for all $i$. Obviously, $s_i$ and $s_{i+1}$ are endpoints of the subsegment $v_i$ of $\gamma_{t,m}$ labeled by $a_mb_m$ or $a_mb_2$ for $0\leq i \leq \ell$-1. Let $T_1$ be the set of indices $i$ such that $v_i$ is labeled by $a_mb_m$ and $T_2$ be the set of indices $i$ such that $v_i$ is labeled by $a_mb_2$. Since each $w_i$ only contains one subword labeled by $a_mb_m$, $T_1$ contains $n$ elements and $T_2$ contains $(l-n)$ elements. For $0\leq i\leq \ell-1$, let $m_i$ be a geodesic with the initial point $s_i$ which runs along $v_i$ followed by $u_{i+1}$. (The fact that $m_i$ is a geodesic is guaranteed by Lemmas~\ref{bl2} and~\ref{bl3}.) 

For each $i$ in $T_1$, we can connect $u_i\bigl(8f_t(n)\bigr)$ and $m_i\bigl(8f_t(n)\bigr)$ by a path $\eta_i$ outside $B\bigl(s_i,8f_t(n)\bigr)$ with length bounded above by $M_1\bigl(f_t(n)\bigr)^{m-1}+N_1$ for some constants $M_1$ and $N_1$ not depending on $r$ and $n$ by Lemma~\ref{bl2}. For each $i$ in $T_2$, we can connect $u_i\bigl(8f_t(n)\bigr)$ and $m_i\bigl(8f_t(n)\bigr)$ by a path $\eta_i$ outside $B\bigl(s_i,8f_t(n)\bigr)$ with length bounded above by $M_2f_t(n)+N_2$ for some constants $M_2$ and $N_2$ not depending on $r$ and $n$ by Lemma~\ref{bl3}. Since the distance between $e$ and $s_i$ is bounded above by $4 f_t(n)$ and $f_t(n)\geq r$, each $\eta_i$ also lies outside $B\bigl(e,r)$ for $0\leq i\leq \ell-1$. Let $\eta_{\ell}$ be a subsegment of $u_{\ell}$ connecting $x$ and $u_{\ell}\bigl(8f_t(n)\bigr)$. Moreover, the ray $\sigma$ with the initial point $e$ which runs along a geodesic segment between $e$ and $s_{\ell}$ followed by $u_{\ell}$ is a geodesic since each pair of consecutive edges of $\sigma_i$ are labeled by two group generators which do not commute. Therefore, $\eta_{\ell}$ lies outside $B(e,r)$. 

For $0\leq i\leq \ell-1$, we have $m_i\bigl(8f_t(n)\bigr)=u_{i+1}\bigl(8f_t(n)-2\bigr)$. Thus, we can connect $m_i\bigl(8f_t(n)\bigr)$ and $u_{i+1}\bigl(8f_t(n)\bigr)$ by a path $\eta'_i$ with length 2. Obviously, $\eta'_i$ lies outside $B(e,r)$.

Let $\eta=(\eta_0\eta'_0)(\eta_1\eta'_1)(\eta_2\eta'_2)\cdots(\eta_{\ell-1}\eta'_{\ell-1})\eta_{\ell}$. Thus, $\eta$ is a path outside $B(e,r)$ connecting $u_0\bigl(8f_t(n)\bigr)$ and $x$. Moreover,
\[\ell(\eta)\leq n\biggl(M_1\bigl(f_t(n)\bigr)^{m-1}+N_1+2\biggr)+(\ell-n)\biggl(M_2f_t(n)+N_2+2\biggr)+8f_t(n). \]
It follows that there is some constant $M_3$ not depending on $r$ and $n$, such that the length of $\eta$ is bounded above by $M_3n^{(m-1)t+1}$. Therefore, there is some constant $M_4$ not depending on $r$ and $n$, such that the length of $\eta$ is bounded above by $M_4r^{(m-1)+\frac{1}{t}}$ 
 
By Lemma~\ref{bl3} and the similar argument as above, we can connect $\gamma_{t,m}(-r)$ and $u_0\bigl(8f_t(n)\bigr)$ by a path $\alpha$ outside $B(e,r)$ with length bounded above by $M_5 r^2+N_5$ for some constants $M_5$ and $N_5$ not depending on $r$. Let $\beta_2= \alpha\eta$. Then $\beta_2$ lies outside $B(e,r)$ and connects $\gamma_{t,m}(-r)$ and $x$. Moreover, the length of $\beta_2$ is bounded above by $M r^{m-1+\frac{1}{t}}$ for some constant $M$ not depending on $r$.
\end{proof}

\begin{prop}
For each integer $m\geq 3$ and real number $t>1$, we have $r^{m-1+\frac{1}{t}} \preceq Div_{\gamma_{t,m}}$ in $X_m$.
\end{prop}

\begin{proof}
For each number $r$ large enough, we can choose an integer $n>n_t$ such that
\[r\leq 10h_t n^t \leq 10f_t(n) \leq 10n^t\leq \biggl(\frac{2}{h_t}\biggr)r.\]

Let $\eta$ be any path outside $B(e,r)$ connecting $\gamma_{t,m}(-r)$ and $\gamma_{t,m}(r)$. Since $\gamma_{t,m}$ restricted to $[-r,r]$ is a geodesic and $\eta$ is a path with the same endpoints, $\eta$ must cross each hyperplane crossed by $\gamma_{t,m}\bigl([-r,r]\bigr)$ at least once. Let $s_0=e$ and $s_i=w_1 w_2 w_3 w_4\cdots w_i$ for $1\leq i\leq \lfloor h_tn\rfloor$. Thus, 
\begin{align*}
\abs{s_i}&= 2\bigl((\lfloor 1^{t-1}\rfloor+1)+(\lfloor 2^{t-1}\rfloor+1)+(\lfloor 3^{t-1}\rfloor+1)+(\lfloor 4^{t-1}\rfloor+1)+\cdots\\&\cdots+(\lfloor i^{t-1}\rfloor+1)\bigr)=2i+2f_t(i)\leq 4i^t.
\end{align*}

Recall that $H_{a_m}$ is the support of the hyperplane that crosses the edge $a_m$ with one endpoint $e$. For $0\leq i\leq \lfloor h_tn\rfloor$, it is not hard to see $\abs{s_i}\leq r$. Thus, the path $\eta$ must cross $s_iH_{a_m}$ for $0\leq i\leq \lfloor h_tn\rfloor$. Let $g_i$ be the point at which $\eta$ first crosses $s_iH_{a_m}$, where $g_i$ lies in the component of the complement of the hyperplane in $s_iH_{a_m}$ containing $e$. Let $u_i$ denote the geodesic connecting $s_i$ and $g_i$ which runs along $s_iH_{a_m}$. Similarly, for $0\leq i\leq \lfloor h_tn\rfloor$, let $h_i$ be the point at which $\eta$ first crosses $s_ia_mb_mH_{a_m}$, where $h_i$ lies in the component of the complement of the hyperplane in $s_ia_mb_mH_{a_m}$ containing $e$. Let $v_i$ denote the geodesic connecting $s_ia_mb_m$ and $h_i$ which runs along $s_ia_mb_mH_{c}$. For $0\leq i\leq \lfloor h_tn\rfloor$, let $\eta_i$ be a subsegment of $\eta$ connecting $g_i$ and $h_{i}$. Let $m_i$ be a geodesic with the initial point $s_i$ which runs along $a_mb_m$ followed by $v_i$. (The fact that $m_i$ is a geodesic is guaranteed by Lemma~\ref{bl3}.) Since
\[d(g_i,s_i)\geq d(g_i,e)-d(s_i,e)\geq r-4i^t\geq 5h_tn^t-4(h_tn)^t\geq h_tn^t\]
and
\[d(h_i,s_i)\geq d(h_i,e)-d(s_i,e)\geq r-4i^t\geq 5h_tn^t-4(h_tn)^t\geq h_tn^t,\]
we have \[\ell(\eta_i)\geq M_1(h_tn^t)^{m-1}-N_1\] for some constants $M_1$ and $N_1$ not depending on $r$ by Lemma~\ref{bl2} and Remark \ref{r1}.
Thus, \[\ell(\eta)\geq \bigl(M_1(h_tn^t)^{m-1}-N_1\bigr)(h_tn-1)\geq M r^{m-1+\frac{1}{t}}\] for some constant $M$ not depending on $r$, which proves the proposition.
\end{proof}

Before showing that each bi-infinite geodesic $\gamma_{t,m}$ is Morse, we would like to mention the concept of lower divergence as follows. The\emph{ lower divergence} of a bi-infinite ray $\gamma$ in a one-ended geodesic space, denoted $\ldiv_{\gamma}$, is the function $h: (0,\infty)\to(0,\infty)$ defined by $h(r)=~ \inf_{t}\rho_{\gamma}(r,t)$, where $\rho_{\gamma}(r,t)$ is the infimum of the lengths of all paths from $\gamma(t-r)$ to $\gamma(t + r)$ which lie outside the open ball of radius $r$ about $\gamma(t)$.

The following proposition characterize Morse geodesics by using the concept of lower divergence.

\begin{prop}[Theorem 2.14, \cite{MR3339446}]
\label{CS}
Let $X$ be a $\CAT(0)$ space and $\gamma \subset X$ a bi-infinite geodesic. Then the bi-infinite geodesic in $X$ is Morse iff it has superlinear divergence.
\end{prop}

\begin{rem}
Let $X$ be a one-ended geodesic space and $\gamma$ a bi-infinite periodic geodesic (i.e. there is some isometry of $X$ acts on $\gamma$ by translation). It is not hard to see that the divergence of $\gamma$ and the lower divergence of $\gamma$ are equivalent.
\end{rem}

\begin{prop}
For each integer $m\geq 3$ and real number $t>1$, the lower divergence of the geodesic ${\gamma_{t,m}}$ is at least quadratic and the geodesic ${\gamma_{t,m}}$ is Morse. 
\end{prop}

\begin{proof}
For each $u$ and each $r$ large enough, let $\eta$ be any path outside $B\bigl(\gamma_{t,m}(u),r\bigr)$ connecting $\gamma_{t,m}(u-r)$ and $\gamma_{t,m}(u+r)$. Since $\gamma_{t,m}{[u-r,u+r]}$ is a geodesic and $\eta$ is a path with the same endpoints, $\eta$ must cross each hyperplane crossed by $\gamma_{t,m}{[u-r,u+r]}$ at least once. Without loss of generality, we can assume $\gamma_{t,m}(u)$ is a group element and $\gamma_{t,m}(u+2)=\gamma_{t,m}(u)a_mb_*$, where $b_*$ is either $b_m$ or $b_2$. 

Let $s_i=\gamma_{t,m}(u+2i)$ for each integer $i$ between $0$ and $r/2$. Recall that $H_{a_m}$ is the support of the hyperplane that crosses the edge $a_m$ with one endpoint $e$. Let $g_i$ be the point at which $\eta$ first crosses $s_iH_{a_m}$, where $g_i$ lies in the component of the complement of the hyperplane in $s_iH_{a_m}$ containing $\gamma_{t,m}(u)$. Let $u_i$ denote the geodesic connecting $s_i$ and $g_i$ which runs along $s_iH_{a_m}$. Let $m_i$ be a geodesic with the initial point $s_i$ which runs along subsegment of $\gamma_{t,m}$ connecting $s_i$, $s_{i+1}$ followed by $u_{i+1}$ for each integer $i$ between $0$ and $r/2-1$. Let $\eta_i$ be a subsegment of $\eta$ connecting $g_i$ and $g_{i+1}$ for each integer $i$ between $0$ and $r/2-1$. 

Since
\[d(g_i,s_i)\geq d\bigl(g_i,\gamma_{t,m}(u)\bigr)-d\bigl(s_i,\gamma_{t,m}(u)\bigr)\geq r-2i\]
and
\[d(g_{i+1},s_i)\geq d\bigl(g_{i+1},\gamma_{t,m}(u)\bigr)-d\bigl(s_i,\gamma_{t,m}(u)\bigr)\geq r-2i,\]
we have \[\ell(\eta_i)\geq M_1(r-2i)-N_1\] for some constants $M_1$ and $N_1$ not depending on $r$ by Lemma~\ref{bl2}, Lemma~\ref{bl3} and Remark \ref{r1}.

Thus,
\begin{align*}
 \ell(\eta)\geq \sum_i \ell(\eta_i)\geq \sum_i \bigl(M_1(r-2i)-N_1\bigr) \geq Mr^2-N
\end{align*}
for some constant $M$ not depending on $r$ and $u$.
Therefore, $\rho_{\gamma_{t,m}}(r,u)$ is bounded below by $Mr^2-N$ for all $u$. Thus, $\gamma_{t,m}$ has at least quadratic lower divergence.
Thus, the geodesic ${\gamma_{t,m}}$ is Morse for each integer $m\geq 3$ and $t>1$ by Proposition \ref{CS}.
\end{proof}

Thus, for each integer $m\geq 3$ and each $s$ in $(m-1,m)$ the geodesic $\gamma_{t,m}$ is Morse and has the divergence function equivalent to $r^s$, where $t=1/(s-m+1)$. Before showing the existence of Morse geodesic with polynomial divergence $r^m$ in each $X_m$, we need some background on Van Kampen diagram.
\begin{defn}[Reduced Van Kampen diagram]
A \emph{reduced Van Kampen diagram} over a group presentation $G=\langle S; R\rangle$, where all $r\in R$ are cyclically reduced words in the free group $F(S)$, is a finite 2-dimensional complex $D$ satisfying the following additional properties:
\begin{enumerate}
\item The complex $D$ is connected and simply connected.
\item Each 1-cell of $D$ is labeled by an arrow and a letter in $S$.
\item Some 0-cell which belongs to the topological boundary of $D$ is specified as a base-vertex.
\item The label of every simple boundary path of a 2-cell of $D$ is an element of $R^*$, where $R^*$ is obtained from $R$ by adding all cyclic permutations of elements of $R$ and of their inverses.
\item The complex $D$ does not contain two 2-cells $B_1$, $B_2$ such that their boundary cycles share a common edge and such that their boundary cycles, read starting from that edge, clockwise for one of the 2-cells and counter-clockwise for the other, are equal as words in $S\cup S^{-1}$.
\end{enumerate} 
\end{defn}

\begin{thm}[\cite{MR0577064}]
\label{vkthm}
Given a group presentation $G=\langle S; R\rangle$, where all $r\in R$ are cyclically reduced words in the free group $F(S)$, and a freely reduced word $w$ in $S \cup S^{-1}$. Then $w=1$ in $G$ iff there exists a reduced Van Kampen diagram $D$ over the presentation $G=\langle S; R\rangle$ whose boundary label is freely reduced and is equal to w.
\end{thm}

We now finish the proof of the Proposition \ref{mp} by the following proposition.

\begin{prop} 
\label{propo}
For each $m\geq 2$, let $\alpha_m$ be a bi-infinite geodesic containing $e$ and labeled by $\cdots a_mb_ma_mb_m\cdots$. Then $\alpha_m$ is a Morse bi-infinite geodesic and the divergence of $\alpha_m$ is equivalent to $r^m$ in $X_m$.
\end{prop}

\begin{figure}
\labellist
\small
\pinlabel $a)$ at 125 520
\pinlabel $r$ at 360 290
\pinlabel $\alpha_m(t)$ at 425 165
\pinlabel $\alpha_m$ at 125 165
\pinlabel $\gamma$ at 250 350

\endlabellist
\centering
\includegraphics[scale=0.18]{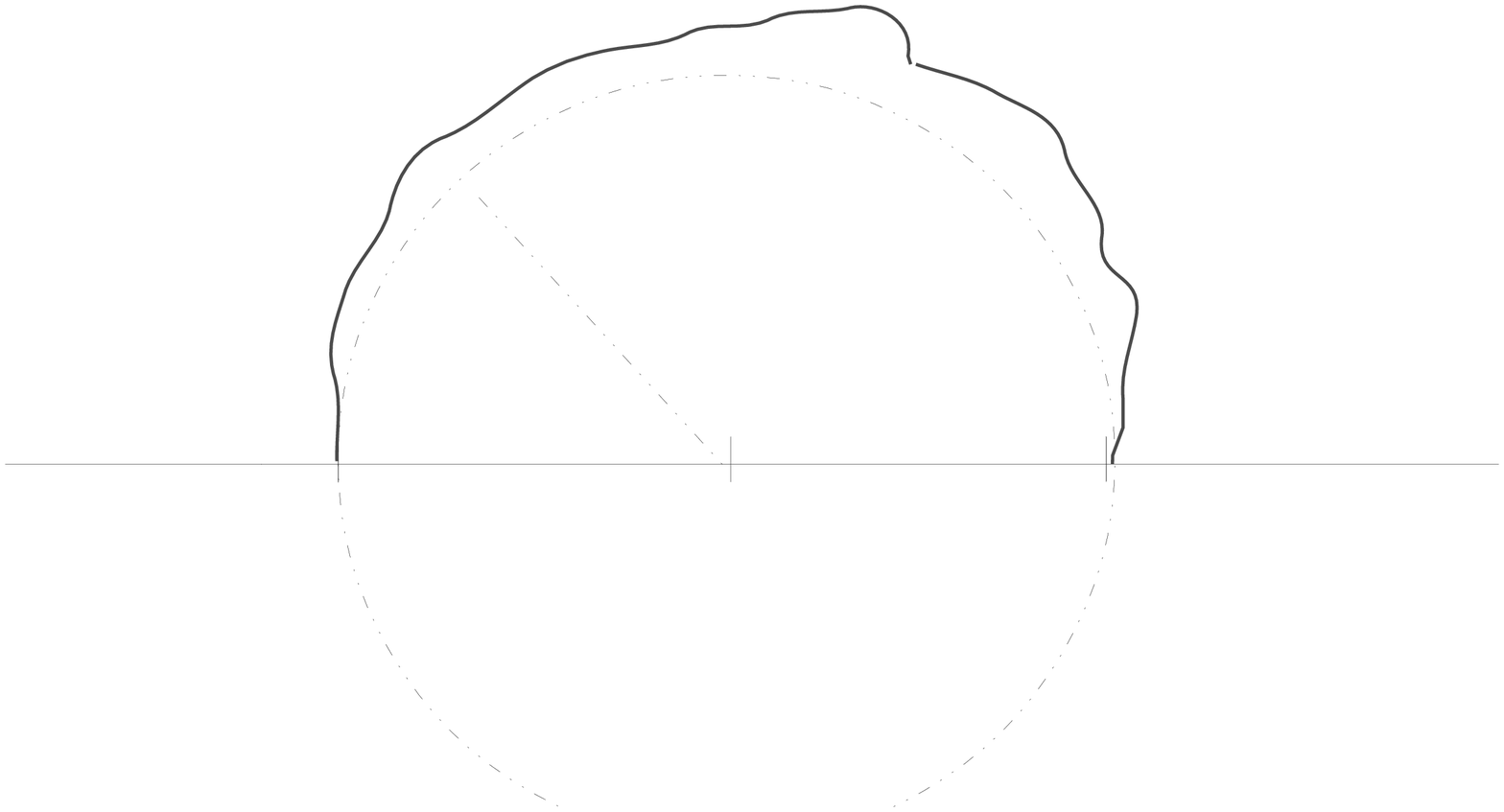}
\labellist
\small
\pinlabel $b)$ at 125 520
\pinlabel $r$ at 360 290
\pinlabel $\alpha_m(t)$ at 425 165
\pinlabel $\alpha_m$ at 125 165
\pinlabel $\gamma$ at 250 350
\pinlabel $\gamma_1$ at 580 350
\pinlabel $\alpha'_m$ at 715 165
\pinlabel $\beta'_m$ at 610 520

\endlabellist
\centering
 \includegraphics[scale=0.18]{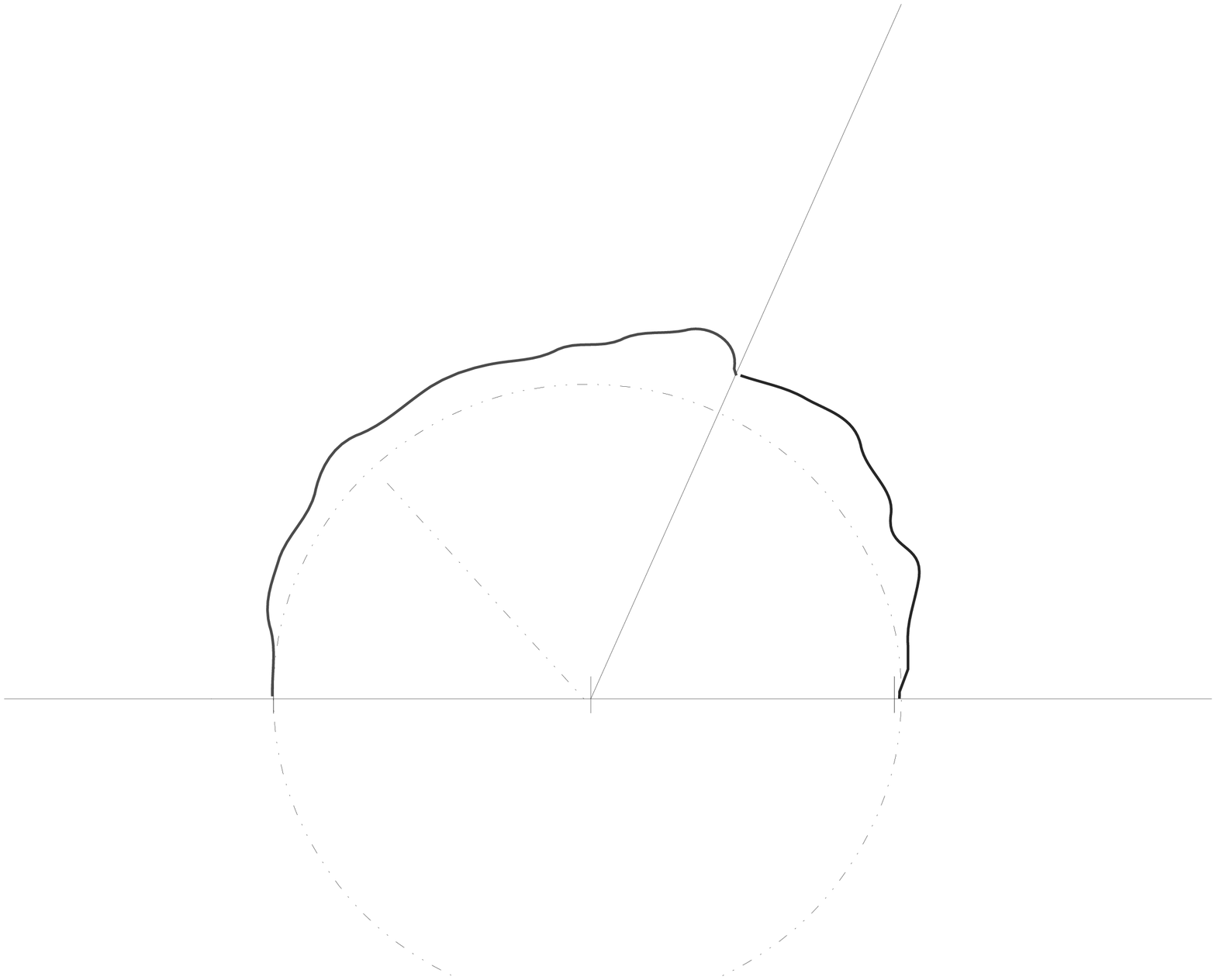}
\caption{The path $\gamma$ lies outside $B\bigl(\alpha_m(t),r\bigr)$ with endpoints $\alpha_m(t-r)$, $\alpha_m(t+r)$ and the subpath $\gamma_1$ of $\gamma$ connecting two points of $\alpha'_m$ and $\beta'_m$.}
\label{ma4}
\end{figure}

\begin{proof}
We first show that the divergence of $\alpha_m$ is equivalent to $r^m$ in $X_m$. Since the divergence of $X_m$ is equivalent to $r^m$ (see Section 5 in \cite{MR3314816}), the divergence of $\alpha_m$ is dominated by $r^m$. We only need to show that $r^m\preceq Div_{\alpha_m}$. 

For each $r>0$, let $\gamma$ be an arbitrary path from $\alpha_m(-r)$ to $\alpha_m(r)$ which lies outside the open ball with radius $r$ about $e$. 

We are going to show that there exists a subpath $\gamma_1$ of $\gamma$ connecting two points of $\alpha'_m$ and $\beta'_m$, where $\alpha'_m$ and $\beta'_m$ are two geodesic rays issuing from $e$, $\alpha'_m$ is labeled by $b_ma_mb_ma_m\cdots$ or $a_mb_ma_mb_m\cdots$ and $\beta'_m$ is labeled by $b_{m-1}a_{m-1}b_{m-1}a_{m-1}\cdots$ or $a_{m-1}b_{m-1}a_{m-1}b_{m-1}\cdots$ (see Figure \ref{ma4}.b). 

We will use the same technique as in \cite{MR1254309} for this argument. We observe that the path $\gamma$ and the subsegment of $\alpha_m$ between $\alpha_m(-r)$, $\alpha_m(r)$ form a loop in $X_m$ which may fill in with a reduced Van Kampen diagram $D$ by Theorem \ref{vkthm}. Here we can assume that each 1-cell of $D$ is unoriented and each 2-cell of $D$ labeled by $s^2$ for some generator $s$ has been shrunk to an unoriented edge with label $s$. We want to obtain a corridor in $D$ that is a concatenation of 2--cells labeled by $b_ma_{m-1}b_ma_{m-1}$ or $b_mb_{m-1}b_mb_{m-1}$ alternately such that the first edge labeled by $b_m$ of the corridor must lie in $\alpha_m$ with one endpoint $e$ and the last edge labeled by $b_m$ of the corridor must lie in the boundary of $D$ (see Figure \ref{ma3}). 

Since the path $\gamma$ lies outside the ball $B(e,r)$, the edge $b_m^{(1)}$ of $\alpha_m$ with one endpoint $e$ and labeled by $b_m$ must lie in some 2--cell of $D$. By the presentation of $G_{\Gamma_m}$, the edge $b_m^{(1)}$ must lie in a 2--cell $c_1$ labeled by $b_ma_{m-1}b_ma_{m-1}$ or $b_mb_{m-1}b_mb_{m-1}$. If the edge $b_m^{(2)}$ that is opposite to $b_m^{(1)}$ in $c_1$ lies in the the boundary of $D$, it is obvious that we can find the corridor we need. Otherwise, $b_m^{(2)}$ must lie in some 2--cell $c_2$ labeled by $b_ma_{m-1}b_ma_{m-1}$ or $b_mb_{m-1}b_mb_{m-1}$ of $D$. Since $D$ is reduced, the labels of $c_1$ and $c_2$ must be different. By arguing inductively, we obtain a corridor in $D$ that is a concatenation of 2--cells labeled by $b_ma_{m-1}b_ma_{m-1}$ or $b_mb_{m-1}b_mb_{m-1}$ alternately such that the first edge labeled by $b_m$ of the corridor lies in $\alpha_m$ with one endpoint $e$ and the last edge $b_m^{(n)}$ labeled by $b_m$ of the corridor lies in the boundary of $D$. If $b_m^{(n)}$ is an edge of $\alpha$, the diagram $D$ would not be planar topologically. Thus, $b_m^{(n)}$ must be an edge of $\gamma$. 

\begin{figure}
\labellist
\small
\pinlabel $r$ at 345 365
\pinlabel $\alpha_m(t)$ at 410 220
\pinlabel $\alpha_m$ at 80 220
\pinlabel $\gamma$ at 215 380
\pinlabel $\gamma_1$ at 630 380
\pinlabel $\alpha'_m$ at 760 220
\pinlabel $\beta'_m$ at 580 525

\pinlabel $b_m$ at 450 275
\pinlabel $b_m$ at 470 320
\pinlabel $b_m$ at 490 365
\pinlabel $b_m$ at 510 410
\pinlabel $b_m$ at 530 455

\pinlabel $a_{m-1}$ at 385 280
\pinlabel $b_{m-1}$ at 405 325
\pinlabel $a_{m-1}$ at 425 370
\pinlabel $b_{m-1}$ at 445 415

\pinlabel $a_{m-1}$ at 525 280
\pinlabel $b_{m-1}$ at 545 325
\pinlabel $a_{m-1}$ at 565 370
\pinlabel $b_{m-1}$ at 585 415

\endlabellist
\centering
\includegraphics[scale=0.3]{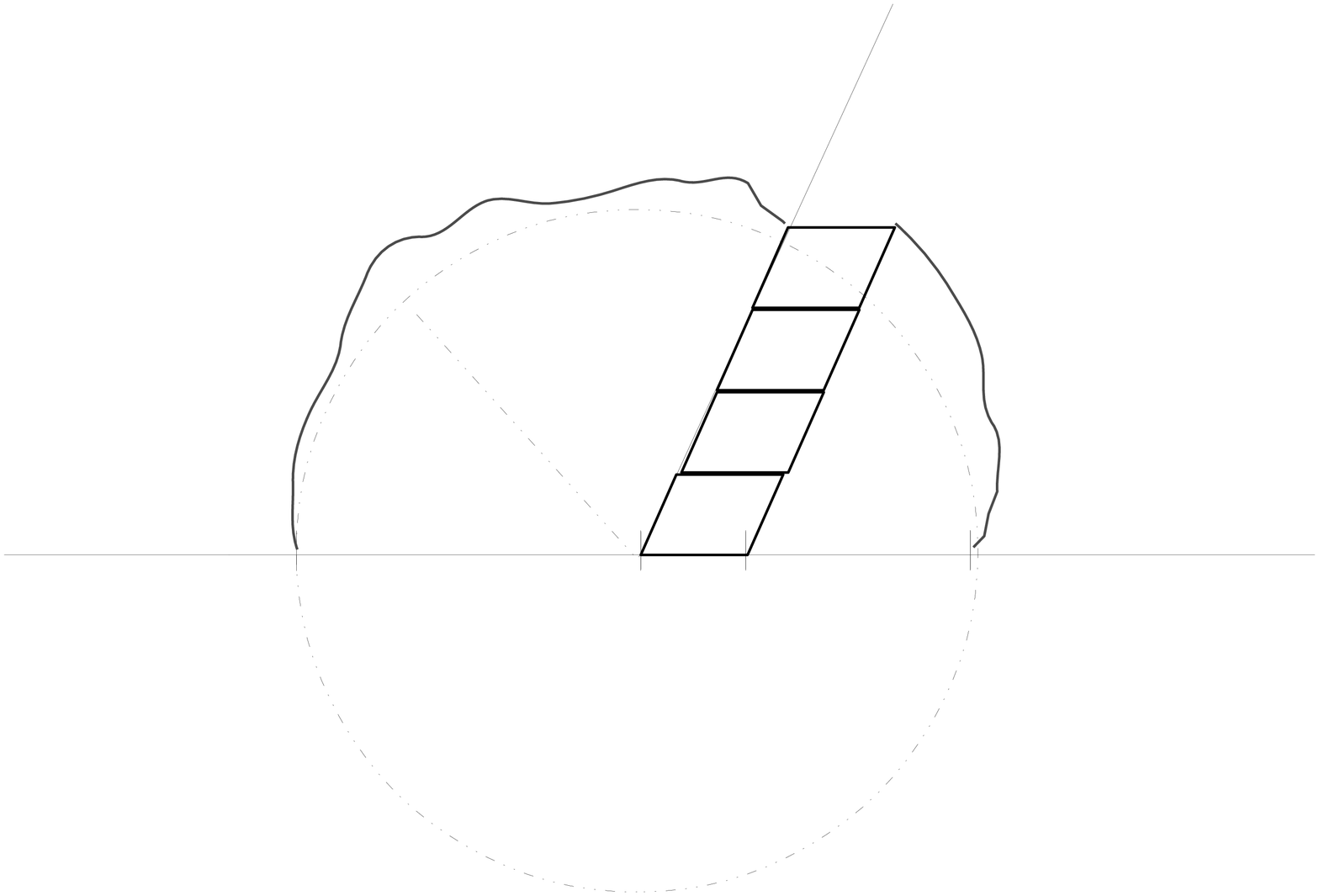}

\caption{}
\label{ma3}
\end{figure}

Therefore, there exists a subsegment $\gamma_1$ of $\gamma$ connecting two points of $\alpha'_m$ and $\beta'_m$, where $\alpha'_m$ and $\beta'_m$ are two geodesic rays issuing from $e$, $\alpha'_m$ is labeled by $b_ma_mb_ma_m\cdots$ or $a_mb_ma_mb_m\cdots$ and $\beta'_m$ is labeled by $b_{m-1}a_{m-1}b_{m-1}a_{m-1}\cdots$ or $a_{m-1}b_{m-1}a_{m-1}b_{m-1}\cdots$. Therefore, the divergence of the pair $(\alpha'_m,\beta'_m)$ is dominated by the divergence $\alpha_m$. Also, the divergence of the pair $(\alpha'_m,\beta'_m)$ is equivalent to $r^m$ (see the proof of Proposition 5.3 in \cite{MR3314816}). Therefore, the divergence of $\alpha_m$ is equivalent to $r^m$. We observe that $\alpha_m$ is a periodic geodesic. Therefore, the lower divergence of $\alpha_m$ and the divergence of $\alpha_m$ are equivalent. Thus, the lower divergence of $\alpha_m$ is super linear. This implies that $\alpha_m$ is a Morse geodesic. 
\end{proof}

We are now ready for the proof of the main theorem. 
\begin{main}
For each integer $m\geq 2$, there is a $\CAT(0)$ space $Y_m$ with a proper, cocompact action of some finitely generated group such that for each $s$ in $[2,m]$ there is a Morse geodesic in $Y_m$ with the divergence function equivalent to $r^s$.
\end{main}
\begin{proof}
For each integer $m\geq 2$ let $\Omega_m$ be the disjoint union of $(m-1)$ graphs $\Gamma_i$ for $i$ in $\{2,3,\cdots,m\}$ and $Y_m$ be the Davis complex with respect to $\Omega_m$. Then the associated right-angled Coxeter group $G_{\Omega_m}=G_{\Gamma_2}*G_{\Gamma_3}*\cdots*G_{\Gamma_m}$ acts properly and cocompactly on $Y_m$. Moreover, there is a bipartite tree $T$ that encodes the structure of $Y_m$ as follows:
\begin{enumerate}
\item There are two types of vertices of $T$ in the bipartition which are called \emph{nontrivial vertices} and \emph{trivial vertices}.
\item Each nontrivial vertex $u$ of $T$ corresponds to a copies $P_u$ (called \emph{nontrivial vertex space}) of $X_i$ (the Davis complex with respect to the graph $\Gamma_i$) in $Y_m$ for some $i$ in $\{2,3,\cdots,m\}$ and each trivial vertex $v$ of $T$ corresponds to a single point $y_v$ in $Y_m$.
\item Two nontrivial vertex spaces $P_u$ and $P_{u'}$ are different if $u$ and $u'$ are different nontrivial vertices in $T$. Similarly, two points $y_v$ and $y_{v'}$ are different if $v$ and $v'$ are different trivial vertices in $T$.
\item A trivial vertex $v$ is adjacent to a nontrivial vertex $u$ in $T$ iff $y_v$ belongs to $P_u$.
\item The union of all nontrivial vertex spaces is equal to $Y_m$ and there exist a nontrivial vertex space which is a copy of $X_i$ for each $i$ in $\{2,3,\cdots,m\}$. Moreover, each nontrivial vertex space is isometric embedding into $Y_m$
\item Two nontrivial vertex spaces $P_u$ and $P_{u'}$ have nonempty intersection iff the two nontrivial vertices $u$ and $u'$ are both adjacent to a trivial vertex $v$ in $T$. Moreover, the intersection of $P_u$ and $P_{u'}$ is a single set consisting of the point $y_v$.
\end{enumerate}

From the above structure of $Y_m$, it is easy to see each simple path connecting two different points in a nontrivial vertex space $P$ must lie in $P$. Let $s$ be an arbitrary number in $[2,m]$. If $s=2$, then let $\gamma_2$ be a Morse bi-infinite geodesic in $X_2$ with divergence function equivalent to $r^2$. Otherwise, let $\gamma_s$ be a Morse bi-infinite geodesic in $X_n$ with divergence function equivalent to $r^s$ where $3\leq n\leq m$ and $s$ in $(n-1,n]$. We can consider the bi-infinite geodesic $\gamma_s$ is contained in some nontrivial vertex space $P$ of $Y_m$ such that $\gamma_s$ is a Morse geodesic in $P$ and the divergence of $\gamma_s$ in $P$ is equivalent to $r^s$. 

We now prove that the divergence of $\gamma_s$ is also equivalent to $r^s$ in $Y_m$. Since the nontrivial vertex space $P$ is isometric embedding into $Y_m$, the divergence of $\gamma_s$ is dominated by $r^s$ in $Y_m$. For each positive number $r$ let $\eta$ be an arbitrary path outside $B(\gamma_s(0),r)$ in $Y_m$ connecting $\gamma_s(-r)$ and $\gamma_s(r)$. Since we can obtain a simple path $\eta'$ from $\eta$ such that $\eta'$ is also a path outside $B(\gamma_s(0),r)$ in $Y_m$ connecting $\gamma_s(-r)$, $\gamma_s(r)$ and the length of $\eta'$ is less than or equal the length of $\eta$, we can assume that $\eta$ is a simple path. By the above observation, the path $\eta$ is contained in $P$. Thus, the divergence of $\gamma_s$ in $P$ is dominated by the divergence of $\gamma_s$ in $Y_m$. Therefore, the divergence of $\gamma_s$ is also equivalent to $r^s$ in $Y_m$.

We now prove that $\gamma_s$ is a Morse geodesic in $Y_m$. Let $\beta$ be an arbitrary $(K,L)$--quasi-geodesic with the endpoints on $\gamma_s$. We need to prove that $\beta$ lies in some $M$--neighborhood of $\gamma_s$ where $M$ only depends on $K$ and $L$. By Lemma 1.11 of \cite{MR1744486} III.H, in , we can assume that $\beta$ is a continuous $(K,L)$--quasi-geodesic. Also it is not hard to see each continuous $(K,L)$--quasi-geodesic lies in some $M_1$-neighborhood of some $(K_1,L_1)$--quasi-geodesic simple path where $M_1$, $K_1$ and $L_1$ only depend on $K$ and $L$. Thus, we can assume that $\beta$ is a $(K,L)$--quasi-geodesic simple path. By the above observation again, the path $\beta$ is contained in $P$. Therefore, the path $\beta$ lies in some $M$--neighborhood of $\gamma_s$ where $M$ only depends on $K$ and $L$ since $\gamma_s$ is a Morse geodesic in $P$. Thus, $\gamma_s$ is a Morse geodesic in $Y_m$.
\end{proof}


\bibliographystyle{alpha}
\bibliography{Tran}
\end{document}